\documentclass[12pt,a4paper,reqno]{amsart}
\usepackage{amsmath,amsthm,amssymb,mathtools,enumitem,setspace,xcolor}
\usepackage[margin=2cm]{geometry}

\usepackage[pagebackref=true,colorlinks=true,citecolor=blue,linkcolor=red,urlcolor=black]{hyperref}

\newtheorem{theorem}{Theorem}[section]
\newtheorem{lemma}[theorem]{Lemma}
\newtheorem{corollary}[theorem]{Corollary}
\newtheorem{proposition}[theorem]{Proposition}

\theoremstyle{definition}
\newtheorem{defn}[theorem]{Definition}
\newtheorem{remark}[theorem]{Remark}

\DeclareMathOperator{\PSL}{{\mathrm{PSL}}}
\DeclareMathOperator{\PSU}{{\mathrm{PSU}}}

\DeclareMathOperator{\PSp}{{\mathrm{PSp}}}

\DeclareMathOperator{\POmega}{{\mathrm{P}\Omega}}

\DeclareMathOperator{\GL}{{\mathrm{GL}}}

\DeclareMathOperator{\Aut}{{\mathrm{Aut}}}
\DeclareMathOperator{\Sym}{{\mathrm{Sym}}}
\DeclareMathOperator{\Inn}{{\mathrm{Inn}}}
\DeclareMathOperator{\Out}{{\mathrm{Out}}}

\makeatletter
\newcommand{\slteq}{%
  \mathrel{\mathpalette\sl@unlhd\relax}%
}

\newcommand{\sl@unlhd}[2]{%
  \sbox\z@{$#1\lhd$}%
  \sbox\tw@{$#1\leqslant$}%
  \dimen@=\ht\tw@
  \advance\dimen@-\ht\z@
  \ifx#1\displaystyle
    \advance\dimen@ .2pt
  \else
    \ifx#1\textstyle
      \advance\dimen@ .2pt
    \fi
  \fi
  \ooalign{\raisebox{\dimen@}{$\m@th#1\lhd$}\cr$\m@th#1\leqslant$\cr}%
}
\makeatother

\renewcommand{\le}{\leqslant}
\renewcommand{\ge}{\geqslant}

\newcommand\sd{\mkern3mu{:}\mkern3mu}

\title[Spreading primitive groups of diagonal type do not exist]{Spreading primitive groups of diagonal type\\do not exist}

\author{John Bamberg}
\author{Saul D. Freedman}
\author{Michael Giudici}

\address{Centre for the Mathematics of Symmetry and Computation, The University of Western Australia, Crawley, WA 6009, Australia}
\email{firstname.lastname@uwa.edu.au}

\begin{document}

\onehalfspacing

\begin{abstract}
The synchronisation hierarchy of finite permutation groups consists of classes of groups lying between $2$-transitive groups and primitive groups. This includes the class of spreading groups, which are defined in terms of sets and multisets of permuted points, and which are known to be primitive of almost simple, affine or diagonal type. In this paper, we prove that in fact no spreading group of diagonal type exists. As part of our proof, we show that all non-abelian finite simple groups, other than six sporadic groups, have a transitive action in which a proper normal subgroup of a point stabiliser is supplemented by all corresponding two-point stabilisers.
\end{abstract}

\maketitle

\section{Introduction}
\label{sec:intro}

An important fact in the characterisation of a finite permutation group $G$ on a finite set $\Omega$ is that all $2$-transitive groups are primitive. However, there is a large gap between primitivity and $2$-transitivity, in the sense that relatively few primitive groups are $2$-transitive. In \cite{synchacs}, Ara\'{u}jo, Cameron and Steinberg introduced a hierarchy of classes of permutation groups that provides a finer characterisation. These classes include \emph{spreading} groups, defined in terms of sets and multisets of elements of $\Omega$ (see \S\ref{sec:prel} for a precise definition); \emph{separating} groups, defined in terms of subsets of $\Omega$, or graphs on $\Omega$ preserved by $G$; and \emph{synchronising} groups, originally defined in \cite{arnold} in terms of transformation semigroups on $\Omega$ containing $G$. (See also \cite{coherent} for equivalent definitions of these properties in terms of $G$-modules.) In particular, these members of the hierarchy are related as follows (see \cite[p.~151]{synchacs}):
\[\text{2-transitive} \implies \text{spreading} \implies \text{separating} \implies \text{synchronising} \implies \text{primitive}.\]
Moreover, the finite primitive groups can be divided into five types via the O'Nan--Scott Theorem, and by \cite[Thereom 2.11 \& Proposition 3.7]{synchacs}, each synchronising group is one of three types: almost simple, affine or diagonal (the relevant diagonal type groups will be defined in \S\ref{sec:prel}). It is also well-known that each finite $2$-transitive group is almost simple or affine.

Now, \cite[\S6--7]{synchacs} classifies the affine spreading groups, and provides examples of almost simple groups that are spreading but not $2$-transitive, separating but not spreading, or synchronising but not separating, while \cite{rank3paper} describes affine synchronising groups that are not spreading. In addition, \cite[Theorem 1.4]{bccsz} states that a primitive group that is not almost simple is synchronising if and only if it is separating. The first (and so far only) known examples of synchronising groups of diagonal type were presented in \cite{bglrdiag}: the groups $\PSL_2(q) \times \PSL_2(q)$ acting diagonally on $\PSL_2(q)$, with $q = 13$ and $q = 17$.

Many open questions involving this hierarchy of permutation groups remain, and in this paper, we work towards solving Problem 12.4 of \cite{synchacs}, the classification of spreading groups (of almost simple or diagonal type). Specifically, we completely resolve the diagonal case, with the aid of the following theorem.

\begin{theorem}[{\cite[Theorem 1.5]{bccsz}}]
\label{thm:synchdiag}
Let $G$ be a synchronising primitive group of diagonal type. Then $G$ has socle $T \times T$, where $T$ is a non-abelian finite simple group.
\end{theorem}

In addition, it is shown in \cite[Theorem 2]{bglrdiag} that the diagonal type group $\PSL_2(q) \times \PSL_2(q)$ is non-spreading for every prime power $q$. We generalise this result to all groups of diagonal type, as follows.

\begin{theorem}
\label{thm:spreaddiag}
Each primitive group of diagonal type is non-spreading.
\end{theorem}

Therefore, the types of spreading groups are precisely the types of finite $2$-transitive groups: almost simple and affine.

For all but six possible socles $T \times T$, we will prove Theorem~\ref{thm:spreaddiag} as a consequence of the following result, which may be of general interest, and which distinguishes six sporadic simple groups from all other non-abelian finite simple groups. For the six remaining socles, we employ elementary character theory.

\begin{theorem}
\label{thm:intersecting}
Let $T$ be a non-abelian finite simple group. Then the following are equivalent.
\begin{enumerate}[label={(\roman*)},font=\upshape]
\item There exists a proper subgroup $A$ of $T$ and a proper normal subgroup $B$ of $A$ such that $A = B(A \cap A^\tau)$ for all $\tau \in \Aut(T)$.
\item There exists a proper subgroup $A$ of $T$ and a proper normal subgroup $B$ of $A$ such that $A = B(A \cap A^t)$ for all $t \in T$.
\item The group $T$ does not lie in the set $\{\mathrm{J}_1, \mathrm{M}_{22}, \mathrm{J}_3, \mathrm{McL}, \mathrm{Th}, \mathbb{M}\}$ of sporadic groups.
\end{enumerate}
\end{theorem}

Property (ii) of this theorem implies that, in the transitive action of $T$ on the set of right cosets of $A$, the proper normal subgroup $B$ of the point stabiliser $A$ is supplemented by all corresponding two-point stabilisers. This is a stronger property than all two-point stabilisers being nontrivial, which is equivalent to the action having base size at least three.

\section{Background}
\label{sec:prel}

In this section, we provide the background on spreading permutation groups, and on primitive groups of diagonal type, that will be necessary to prove Theorems~\ref{thm:spreaddiag} and \ref{thm:intersecting}.

Given a set $\Omega$, a multiset (or set) $J$ of elements of $\Omega$ and a point $\omega \in \Omega$, we write $\mu_J(\omega)$ to denote the multiplicity of $\omega$ in $J$. We say that $J$ is \emph{trivial} if either $\mu_J$ is constant on $\Omega$, or there exists a unique point $\omega \in \Omega$ such that $\mu_J(\omega) \ne 0$. Additionally, a sum $J+K$ of multisets is defined so that $\mu_{J+K}(\omega) = \mu_J(\omega) + \mu_K(\omega)$ for all $\omega$, and similarly for scalar products of multisets. As usual, the cardinality of $J$ is $|J| := \sum_{\omega \in \Omega} \mu_J(\omega)$.

We are now able to define spreading permutation groups.

\begin{defn}[{\cite[\S5.5]{synchacs}}]
\label{defn:spreading}
A transitive permutation group $G$ on $\Omega$ is \emph{non-spreading} if there exist a nontrivial subset $X$ of $\Omega$ and a nontrivial multiset $J$ of elements of $\Omega$, such that $|J|$ divides $|\Omega|$ and $\sum_{x \in X^g} \mu_J(x)$ is constant for all $g \in G$. Otherwise, $G$ is \emph{spreading}.
\end{defn}

We shall call a pair $(X,J)$ satisfying the above properties a \emph{witness} to $G$ being non-spreading. Recall from \S\ref{sec:intro} that each imprimitive group is non-spreading. It is also immediate from Definition~\ref{defn:spreading} that if $G$ is non-spreading, then so is each transitive subgroup of $G$. The following theorem provides a useful method for finding a witness to a group being non-spreading.

\begin{theorem}[{\cite{polarspaces}}]
\label{thm:ABlemma}
Let $G$ be a finite group acting transitively on a finite set $\Omega$, let $A$ be a subgroup of $G$, and let $B$ a proper normal subgroup of $A$. Additionally, let $\omega_1, \ldots, \omega_k \in \Omega$ with $k\geqslant 2$, such that $\omega_i^B \neq \omega_j^B$ for $i \neq j$, and $\omega_1^A = \omega_1^B \cup \omega_2^B \cup \ldots \cup \omega_k^B$. Finally, let $\Delta := \{Y \in X^G \mid Y \cap \omega_1^A \neq \varnothing\}$ for some non-empty $X \subsetneq \Omega$. If $B$ is transitive on each $A$-orbit of $\Delta$, then $(X,\Omega + k\omega_1^B - \omega_1^A)$ is a witness to $G$ being non-spreading.
\end{theorem}

Now, let $T$ be a non-abelian finite simple group. In the remainder of this section, we shall consider the primitive groups of diagonal type with socle $T \times T$. To define these groups, let $T \times T$ act on the set $\Omega:=T$ by $x^{(t_1,t_2)}:=t_1^{-1}xt_2$. In addition, let $\rho_1:T\times T \to \Sym(\Omega)$ be the corresponding permutation representation, and let $\rho_2:\Aut(T) \to \Sym(\Omega)$ be the natural permutation representation of $\Aut(T)$ on $\Omega$. Then $\rho_2(\Inn(T))=\rho_1(\{(t,t)\mid t\in T\})$. Finally, let $\sigma\in\Sym(\Omega)$ such that $x^\sigma=x^{-1}$ for all $x\in T$. Define $W(T):=\langle \rho_1(T\times T),\rho_2(\Aut(T)),\sigma\rangle\leqslant\Sym(\Omega)$, and note that $W(T)\cong (T\times T).(\Out(T) \times C_2)$. The \emph{groups of diagonal type} with socle $T \times T$ are precisely the subgroups of $W(T)$ containing $\rho_1(T\times T)$ (see, for example, \cite[\S1]{lps}). We will often drop the use of $\rho_1$ and $\rho_2$, and allow elements (and subgroups) of $T\times T$ and $\Aut(T)$ to denote permutations (and subgroups) of $\Sym(\Omega)$.

In order to derive Theorem~\ref{thm:spreaddiag} from Theorem~\ref{thm:intersecting} (for most non-abelian finite simple groups), we require the following corollary of Theorem~\ref{thm:ABlemma}. Throughout the remainder of this section, we will use the elementary fact that if $G$ is a transitive permutation group on a set $\Sigma$ and $H\leqslant G$, then $H$ acts transitively on $\Sigma$ if and only if $G=HG_\alpha$ for each $\alpha\in\Sigma$.

\begin{corollary}
\label{cor:diagonalaction}
Let $T$ be a non-abelian finite simple group, and suppose that there exists a proper subgroup $A$ of $T$ and a proper normal subgroup $B$ of $A$ such that $A = B(A \cap A^\tau)$ for all $\tau \in \Aut(T)$. Then $(A,\Omega+|A:B|B-A)$ is a witness to the diagonal type group $W(T)$ being non-spreading in its action on $\Omega = T$.
\end{corollary}

\begin{proof}
Let $G:=W(T)$, $M:=\rho_1(1\times T)$, $\omega_1:=1_T$, and $D:=G_{\omega_1} = \langle \rho_2(\Aut(T)),\sigma\rangle\cong\Aut(T)\times C_2$. Note that $G=DM$, as $M$ acts transitively on $\Omega$. Let $X:=A\subsetneq \Omega$. Then $\omega_1^{1 \times B} = B < A = \omega_1^{1 \times A}$, and $X^{G}=X^{DM}=\{A^\tau t\mid \tau\in \Aut(T),t\in T\}$. Next, define $\Delta:={\{Y\in X^G\mid Y\cap \omega_1^{1\times A}\neq \varnothing\}}$, and observe that $\Delta = \{A^\tau a\mid \tau\in \Aut(T),a\in A\}$. Thus the orbits of $1 \times A$ on $\Delta$ are $\Delta_\tau=\{A^\tau a\mid a\in A\}$ for each $\tau\in \Aut(T)$. Additionally, the stabiliser in $1\times A$ of $A^\tau\in \Delta_\tau$ is $1\times (A\cap A^\tau)$. As $A=B(A\cap A^\tau)$, we observe that $1\times B$ is transitive on $\Delta_\tau$. The result now follows from Theorem~\ref{thm:ABlemma}, applied to the groups $G$, $1 \times A$ and $1 \times B$. 
\end{proof}

Our next lemma translates Property (i) of Theorem~\ref{thm:intersecting} into the language of coset actions.

\begin{lemma}
\label{lem:cosetaction}
Let $T$ be a non-abelian finite simple group, $R \in \{T,\Aut(T)\}$, and $A$ a proper subgroup of $T$. Suppose also that $A$ has a proper normal subgroup $B$ that acts transitively on each orbit of $A$ in its action on the right cosets of $A$ in $R$. If $R = \Aut(T)$, or if all $\Aut(T)$-conjugates of $A$ are conjugate in $T$, then $A = B(A \cap A^\tau)$ for all $\tau \in \Aut(T)$.
\end{lemma}

\begin{proof}
Let $\Sigma$ be the set of right cosets of $A$ in $R$. Then the point stabilisers in the action of $A$ on $\Sigma$ are precisely the subgroups $A \cap A^t$ with $t \in R$. As $B$ is transitive on each orbit in this action, it follows that $A = B(A \cap A^t)$ for all $t \in R$. Hence we are done if $R = \Aut(T)$. Assume therefore that all $\Aut(T)$-conjugates of $A$ are conjugate in $T$. Then for each $\tau \in \Aut(T)$, there exists $t \in T$ such that $A^\tau = A^t$, and so $B(A \cap A^\tau) = B(A \cap A^t) = A$.
\end{proof}

In the case where $T$ is one of the six sporadic groups listed in Property (iii) of Theorem~\ref{thm:intersecting}, we will show that $W(T)$ is non-spreading using the following lemma, which involves elementary character theory. The multiset in the witness here is similar to that of Theorem~\ref{thm:ABlemma}.

\begin{lemma}
\label{lem:chars}
Let $T$ be a non-abelian finite simple group, and let $\mathrm{Irr}(T)$ be the set of irreducible complex characters of $T$. Suppose also that there exist pairwise non-conjugate elements $r$, $s_1$ and $s_2$ of $T$, such that:
\begin{enumerate}[label={(\roman*)},font=\upshape]
\item $|s_1^T| = |s_2^T|$; and
\item $\chi(r^\tau) = 0$ for all $\tau \in \Aut(T)$ and all $\chi \in \mathrm{Irr}(T)$ with $\chi(s_1) \ne \chi(s_2)$.
\end{enumerate}
Then $(r^T,\Omega+s_1^T-s_2^T)$ is a witness to the diagonal type group $W(T)$ being non-spreading in its action on $\Omega = T$.
\end{lemma}

\begin{proof}
The set $X:=r^T$ and the multiset $J:=\Omega+s_1^T-s_2^T$ are clearly nontrivial, and (i) implies that $|J| = |\Omega|$. Thus by Definition~\ref{defn:spreading}, it remains to prove that $\sum_{x \in X^g} \mu_J(x)$ is constant for all $g \in G:=W(T)$. In fact, we will show that $|X^g \cap s_1^T| = |X^g \cap s_2^T|$ for all $g$, and it will follow immediately that $\sum_{x \in X^g} \mu_J(x) = |X|$ for all $g$.

As above, $G_{1_T} = \langle \rho_2(\Aut(T)),\sigma\rangle$, and the transitivity of $M:=\rho_1(1 \times T)$ on $\Omega$ implies that $G = G_{1_T}M$. Hence for each $g \in G$, there exist $\tau \in \Aut(T)$, $\varepsilon = \pm 1$ and $t \in T$ such that $X^g = ((r^\tau)^{\varepsilon})^Tt$. Letting $m:=(r^\tau)^\varepsilon$, $i \in \{1,2\}$ and $h \in s_i^T$, we observe from \cite[Problem 3.9]{isaacs} that the set $\{(x,y) \mid x \in m^T, y \in t^T, xy = h\}$ has size \[a_{g,i}:=\frac{|m^T||t^T|}{|T|}\sum_{\chi \in \mathrm{Irr}(T)}\frac{\chi(m)\chi(t)\overline{\chi(s_i)}}{\chi(1)},\] where $\overline{\chi(s_i)}$ is the complex conjugate of $\chi(s_i)$. Since $a_{g,i}$ is constant for all $h \in s_i^T$, and $|m^Tt^u \cap s_i^T|$ is constant for all $u \in T$, it follows that $|X^g \cap s_i^T| = a_{g,i}|s_i^T|/|t^T|$. As $\chi((r^\tau)^{-1}) = \overline{\chi(r^\tau)}$ for each character $\chi$ of $T$, (i) and (ii) yield $|X^g \cap s_1^T| = |X^g \cap s_2^T|$ for all $g \in G$, as required.
\end{proof}

\begin{remark}
For many small non-abelian finite simple groups $T$, it is possible to choose an element $r \in T$ satisfying Property (ii) of Lemma~\ref{lem:chars} when $s_1$ and $s_2$ are non-conjugate elements of $T$ such that $s_1^{-1} = s_2$, or more generally, such that $\langle s_1 \rangle = \langle s_2 \rangle$. The non-abelian finite simple groups with all elements conjugate to their inverses are classified in \cite[Theorem 1.2]{realelts}, and \cite[Corollary B.1]{rationalelts} shows that $\PSp_6(2)$ and $\POmega^+_8(2)$ are the only such groups where $u,v \in T$ are conjugate whenever $\langle u \rangle = \langle v \rangle$. Denoting conjugacy classes as in the \textsc{Atlas} \cite{ATLAS}, we observe by inspecting the \textsc{Atlas} character tables of these two groups that we can choose $r \in \mathrm{7A}$, $s_1 \in \mathrm{6A}$ and $s_2 \in \mathrm{6B}$ when $T = \PSp_6(2)$, and $r \in \mathrm{7A}$, $s_1 \in \mathrm{15A}$ and $s_2 \in \mathrm{15B}$ when $T = \POmega^+_8(2)$. We leave open the problem of determining which non-abelian finite simple groups, if any, do not contain pairwise non-conjugate elements $r$, $s_1$ and $s_2$ satisfying Properties (i) and (ii) of Lemma~\ref{lem:chars}.
\end{remark}

\section{Non-abelian finite simple groups and groups of diagonal type}
\label{sec:simple}

In this section, we show that Property (i) of Theorem~\ref{thm:intersecting} holds for all non-abelian finite simple groups other than the sporadic groups $\mathrm{J}_1$, $\mathrm{M}_{22}$, $\mathrm{J}_3$, $\mathrm{McL}$, $\mathrm{Th}$ and $\mathbb{M}$, and then prove Theorems~\ref{thm:intersecting} and \ref{thm:spreaddiag}. Note that we address the Tits group ${}^2F_{4}(2)'$ together with the sporadic groups. Throughout, we use \textsc{Atlas} \cite{ATLAS} notation for the structures of groups.

We first consider the finite simple groups of Lie type, beginning with the Chevalley groups (i.e.~the untwisted groups of Lie type) defined over fields of size $q \ge 3$, then the twisted groups of Lie type, and finally the Chevalley groups with $q = 2$.

\begin{proposition}
\label{prop:chevalleylarge}
Let $T$ be a finite simple Chevalley group $X_\ell(q)$, with $q \ge 3$. Then Property \upshape{(i)} of Theorem~\ref{thm:intersecting} holds for $T$.
\end{proposition}

\begin{proof}
Let $p$ be the prime dividing $q$, and let $U$ be a Sylow $p$-subgroup of $T$. Then $T$ contains a subgroup $N$ such that $A:=N_T(U)$ and $N$ form a $(B,N)$-pair, and the normal subgroup $H:=A \cap N$ of $N$ complements $U$ in $A$ (see \cite[Ch.~7.2--9.4]{carter}). Since $q \ge 3$ (and since $T$ is not the soluble group $A_1(3) = \PSL_2(3)$), we observe from \cite[pp.~121--122]{carter} that $|H| > 1$, and so $U$ is a proper subgroup of $A$. Furthermore, as $A$ is the normaliser in $T$ of the Sylow subgroup $U$, all $\Aut(T)$-conjugates of $A$ are conjugate in $T$.

By Lemma~\ref{lem:cosetaction}, it suffices to show that $A$ and $U$ have the same orbits in the action of $T$ on the set of right cosets of $A$. For each element $w$ of the Weyl group $W:=N/H$, fix  a preimage $n_w \in N$ of $w$ under the natural homomorphism from $N$ to $W$. In addition, let $x$ and $y$ be elements of $T$ so that $Ax$ and $Ay$ lie in a common $A$-orbit. Then the double cosets $AxA$ and $AyA$ are equal. By \cite[Theorem 8.4.3]{carter}, there exist $a_x,a_y \in A$, $w_x,w_y \in W$ and $u_x,u_y \in U < A$ such that $x = a_x n_{w_x} u_x$ and $y = a_y n_{w_y} u_y$. It follows that $A n_{w_x} A = A n_{w_y} A$. As $A$ and $N$ form a $(B,N)$-pair, \cite[Proposition 8.2.3]{carter} yields $n_{w_x} = n_{w_y}$. Thus $AxU = An_{w_x}U = An_{w_y}U = AyU$, i.e.~$Ax$ and $Ay$ lie in the same $U$-orbit, as required.
\end{proof}

\begin{proposition}
\label{prop:twistedlie}
Let $T$ be a finite simple twisted group of Lie type ${}^tX_\ell(q)$. Then Property \upshape{(i)} of Theorem~\ref{thm:intersecting} holds for $T$.
\end{proposition}

\begin{proof}
Note that $T \not\cong {}^2F_4(2)'$. As above, let $p$ be the prime dividing $q$, and let $U$ be a Sylow $p$-subgroup of $T$. By \cite[Ch.~13--14]{carter}, $T$ contains subgroups $N$, $H$ and $W$ analogous to those in the proof of Proposition~\ref{prop:chevalleylarge}, so that $A:=U \sd H$ and $N$ form a $(B,N)$-pair, and $W = N/H$. Furthermore, $H$ is nontrivial for all $q$, as $T$ is not isomorphic to ${}^2A_2(2) = \PSU_3(2)$, ${}^2B_2(2) = \mathrm{Sz}(2)$ or ${}^2F_4(2)$. Although \cite{carter} directly states only that $A \le N_T(U)$, we can show that $A = N_T(U)$. Indeed, if this were not the case, then some preimage $n_w \in N$ of a \emph{fundamental reflection} $w \in W$ would normalise $U$ (see \cite[Proposition 8.2.2 \& Theorem 8.3.2]{carter}). However, we deduce from Proposition 13.6.1 and the proof of Theorem 13.5.4 in \cite{carter} that no such $w$ exists, and so $A = N_T(U)$. We now proceed exactly as in the proof of Proposition~\ref{prop:chevalleylarge}, this time using Proposition 13.5.3 of \cite{carter} instead of Theorem 8.4.3.
\end{proof}

\begin{proposition}
\label{prop:chevalleysmall}
Let $T$ be a finite simple Chevalley group $X_\ell(2)$. Then Property \upshape{(i)} of Theorem~\ref{thm:intersecting} holds for $T$.
\end{proposition}

\begin{proof}
Since $A_2(2) = \PSL_3(2) \cong \PSL_2(7)$ is addressed by Proposition~\ref{prop:chevalleylarge}, and since the remaining groups $X_\ell(2)$ with $\ell \le 2$ are not simple, we shall assume that $\ell \ge 3$. Let $\Phi \subseteq \mathbb{R}^\ell$ be a root system for $T$, with $\Phi^+ \subseteq \Phi$ a system of positive roots and $\Pi \subseteq \Phi^+$ a system of simple roots (see, for example, \cite[Ch.~2]{carter}). Then $\Phi$ is the disjoint union of $\Phi^+$ and $-\Phi^+$, and each node in the Dynkin diagram $D$ corresponding to $T$ is associated with a unique root in $\Pi$ (and vice versa). Hence each diagram automorphism of $D$ induces a permutation of $\Pi$. Let $r$ and $s$ be the simple roots associated with a leaf in $D$ and the adjacent node, respectively. Additionally, let $S$ be the closure of the set $\{s\}$ under the group of diagram automorphisms of $D$, so that $|S| \in \{1,2\}$.

Next, let $U$, $N$ and $W$ be as in the proof of Proposition~\ref{prop:chevalleylarge}. In this case, the group $H$ from that proof is trivial, and so $N_T(U) = U$ and $W \cong N$. By \cite[Proposition 2.1.8, p.~68 \& p.~93]{carter}, $T$ is generated by a set $\{x_u \mid u \in \Phi\}$ of involutions ($x_u$ is written as $x_u(1)$ in \cite{carter}), with $U = \langle x_u \mid u \in \Phi^+ \rangle$, and $N$ is generated by a set $\{n_v \mid v \in \Pi\}$ of involutions. Let $A := \langle U, n_v \mid v \in \Pi \setminus S\rangle$ be the parabolic subgroup of $T$ corresponding to the subset $\Pi \setminus S$ of $\Pi$. Then by \cite[Ch.~8.5]{carter} and the well-known structure of parabolic subgroups of Chevalley groups, $A$ has shape $[2^m]:(L_{\{r\}} \times L_{\Pi \setminus (\{r\} \cup S)})$ for some positive integer $m$, where $L_{\{r\}} \cong A_1(2) \cong S_3$ and $L_{\Pi \setminus (\{r\} \cup S)}$ are Levi subgroups. Since $q = 2$, all automorphisms of $T$ are products of inner and graph automorphisms. Moreover, $\Pi \setminus S$ is fixed setwise by all diagram automorphisms of $D$, which correspond to the graph automorphisms of $T$, and so all $\Aut(T)$-conjugates of $A$ are conjugate in $T$. In addition, $A$ contains an index two subgroup $B$ of shape $[2^m]:(F \times L_{\Pi \setminus (\{r\} \cup S)})$, where $F := L_{\{r\}}' \cong C_3$. We will show that $AtA = AtB$ for all $t \in T$, and the result will follow from Lemma~\ref{lem:cosetaction}.

Since $N_T(U) = U \le A$, we deduce as in the proof of Proposition~\ref{prop:chevalleylarge} that, for each $t \in T$, there exists $n \in N$ such that $AtA = AnA$. Thus it suffices to prove that $AnA = AnB$ for all $n \in N$ (it will immediately follow that $AnB = AtB$ for the corresponding $t \in T$). We also observe using \cite[Ch.~8.5]{carter} (and the fact that $F$ has index two in $L_{\{r\}}$) that $x_r \in A \setminus B$ and $x_{-r} x_r \in B$. The former containment implies that $AnA = AnB \cup Anx_rB$. To complete the proof, we will show that $Anx_rB = AnB$.

The Weyl group $W$ acts linearly on $\mathbb{R}^\ell$ and fixes $\Phi$. Let $w$ be the element of $W$ corresponding to $n$. By \cite[Lemma 7.2.1(i)]{carter}, $nx_r = x_{w(r)} n$ and $x_{-w(r)} n = x_{w(-r)} n = nx_{-r}$. If $w(r) \in \Phi^+$, then $x_{w(r)} \in \langle x_u \mid u \in \Phi^+ \rangle = U \le A$, and so $Anx_r = An$. Otherwise, $-w(r) \in \Phi^+$ and $x_{-w(r)} \in A$, yielding $An = An x_{-r}$ and hence $Anx_r = An x_{-r} x_r$, which lies in $AnB$ by the previous paragraph. In either case, we see that $Anx_rB = AnB$, as required.
\end{proof}

Next, we consider the alternating groups, followed by the sporadic groups and the Tits group.

\begin{proposition}
\label{prop:altgps}
Let $T$ be a finite simple alternating group $A_n$. Then Property \upshape{(i)} of Theorem~\ref{thm:intersecting} holds for $T$.
\end{proposition}

\begin{proof}
Since $A_5 \cong \PSL_2(4) = A_1(4)$ and $A_6 \cong \PSL_2(9) = A_1(9)$ are addressed by Proposition~\ref{prop:chevalleylarge}, we shall assume that $n \ge 7$. Let $\Sigma$ be the set of $3$-subsets of a set of size $n$, and let $\alpha\in\Sigma$. Then $A:=T_\alpha=(S_3\times S_{n-3})\cap T$, and all $\Aut(T)$-conjugates of $A$ are conjugate in $T$. Additionally, $A$ has four orbits on $\Sigma$, namely $\{\alpha\}$ and $\{\beta \in \Sigma \mid |\beta\cap \alpha|=i\}$ for $i \in \{0,1,2\}$. Since the index two subgroup $B:=A_3 \times A_{n-3}$ of $A$ acts transitively on each of these orbits, and since the action of $T$ on $\Sigma$ is equivalent to its action on the right cosets of $A$, the result follows from Lemma~\ref{lem:cosetaction}.
\end{proof}

\begin{proposition}
\label{prop:sporadicgps}
Let $T$ be the Tits group ${}^2F_{4}(2)'$, or a sporadic simple group that does not lie in $\{\mathrm{J}_1, \mathrm{M}_{22}, \mathrm{J}_3, \mathrm{McL}, \mathrm{Th}, \mathbb{M}\}$. Then Property \upshape{(i)} of Theorem~\ref{thm:intersecting} holds for $T$.
\end{proposition}

\begin{proof}
We observe from the \textsc{Atlas} \cite{ATLAS} that $T$ has a maximal subgroup $A$, as specified in Table~\ref{table:sporadics}, such that either $T = \mathrm{O'N}$ or all $\Aut(T)$-conjugates of $T_\alpha$ are conjugate in $T$. It is also clear that the group $B$ in the table is a proper normal subgroup of $A$. Note that if $T = \mathrm{M}_{12}$, then $A$ has two normal subgroups isomorphic to $S_5$; in what follows, $B$ may be chosen as either.

Now, let $R:=\Aut(T)$ if $T = \mathrm{O'N}$, and $R:=T$ otherwise. Additionally, let $\chi$ be the permutation character corresponding to the action of $R$ on the set $\Sigma$ of right cosets of $A$ in $R$. Then for each $t \in R$, the number of points in $\Sigma$ fixed by $t$ is equal to $\chi(t)$. By the Cauchy--Frobenius Lemma, $A$ and $B$ have $c_A:=\frac{1}{|A|}\sum_{a \in A} \chi(a)$ and $c_B:=\frac{1}{|B|}\sum_{b \in B} \chi(b)$ orbits on $\Sigma$, respectively. It is straightforward to compute $c_A$ and $c_B$ in GAP \cite{GAP} using the Character Table Library \cite{GAPchar} ($c_A$ is most readily calculated as the inner product $[\chi,\chi]$), and in each case we obtain $c_A = c_B$. Thus Lemma~\ref{lem:cosetaction} yields the result.
\end{proof}

\begin{table}
\centering
\renewcommand{\arraystretch}{1.2}
\caption{A maximal subgroup $A$ of a group $T$ from Proposition~\ref{prop:sporadicgps}, with $B$ a proper normal subgroup of $A$.}
\label{table:sporadics}
\begin{tabular}{ cccc }
\hline
$T$ & $A$ & $B$ & $|T:A|$ \\
\hline
$\mathrm{M}_{11}$ & $A_6.2$ & $A_6$ & $11$ \\
$\mathrm{M}_{12}$ & $S_5 \times 2$  & $S_5$ & $396$ \\
$\mathrm{J}_2$ & $A_5 \times D_{10}$ & $A_5 \times 5$ & $1008$ \\
$\mathrm{M}_{23}$ & $(2^4 \sd (3 \times A_5)) \sd 2$  & $2^4 \sd (3 \times A_5)$ & $1771$ \\
${}^2F_{4}(2)'$ & $(2^2.[2^8]) \sd S_3$ & $(2^2.[2^8]) \sd 3$ & $2925$ \\
$\mathrm{HS}$ & $S_8$ & $A_8$ & $1100$ \\
$\mathrm{M}_{24}$ & $\mathrm{M}_{12} \sd 2$ & $\mathrm{M}_{12}$ & $1288$ \\
$\mathrm{He}$ & $\PSp_4(4) \sd 2$ & $\PSp_4(4)$ & $2058$ \\
$\mathrm{Ru}$ & $(2^6 \sd \mathrm{PSU}_3(3)) \sd 2$ & $2^6 \sd \mathrm{PSU}_3(3)$ & $188500$ \\
$\mathrm{Suz}$ & $(3.\mathrm{PSU}_4(3)) \sd 2$ & $3.\mathrm{PSU}_4(3)$ & $22880$ \\
$\mathrm{O'N}$ & $\PSL_3(7) \sd 2$ & $\PSL_3(7)$ & $122760$ \\
$\mathrm{Co}_3$ & $\mathrm{McL} \sd 2$ & $\mathrm{McL}$ & $276$ \\
$\mathrm{Co}_2$ & $\mathrm{PSU}_6(2) \sd 2$ & $\mathrm{PSU}_6(2)$ & $2300$ \\
$\mathrm{Fi}_{22}$ & $\POmega_8^+(2) \sd S_3$ & $\mathrm{P}\Omega_8^+(2) \sd 3$ & $61776$ \\
$\mathrm{HN}$ & $\mathrm{PSU}_3(8) \sd 3$ & $\mathrm{PSU}_3(8)$ & $16500000$ \\
$\mathrm{Ly}$ & $(3.\mathrm{McL})\sd 2$ & $3.\mathrm{McL}$ & $9606125$\\
$\mathrm{Fi}_{23}$ & $\POmega_8^+(3) \sd S_3$ & $\POmega_8^+(3) \sd 3$ & $137632$ \\
$\mathrm{Co}_1$ & $(3.\mathrm{Suz}) \sd 2$ & $3.\mathrm{Suz}$ & $1545600$ \\
$\mathrm{J_4}$ & $(2_+^{1+12}.3\mathrm{M}_{22}) \sd 2$ & $2_+^{1+12}.3\mathrm{M}_{22}$ & $3980549947$\\
$\mathrm{Fi}_{24}'$ & $(3 \times \POmega^+_8(3) \sd 3) \sd 2$ & $3 \times \POmega^+_8(3) \sd 3$ & $14081405184$\\
$\mathbb{B}$ & $(2.{}^2E_6(2)) \sd 2$ & $2.{}^2E_6(2)$ & $13571955000$\\
\hline
\end{tabular}
\end{table}

We are now able to prove Theorem~\ref{thm:intersecting}. Given a group $G$ and a non-trivial (core-free) subgroup $H$ of $G$, we will write $b(G,H)$ to denote the base size of $G$ in its action on the set $\Sigma$ of right cosets of $H$, i.e., the minimum size of a subset $\Delta$ of $\Sigma$ such that the pointwise stabiliser $G_{(\Delta)}$ is trivial. Notice that $b(G,H) \ge 2$, as $H$ is a point stabiliser in this action.

\begin{proof}[Proof of Theorem~\ref{thm:intersecting}]
Propositions~\ref{prop:chevalleylarge}--\ref{prop:sporadicgps} show that (iii) implies (i), which clearly implies (ii). To complete the proof, we will show that if (iii) does not hold, then neither does (ii). If $T \in \{\mathrm{J}_1, \mathrm{M}_{22}, \mathrm{J}_3, \mathrm{McL}\}$, then we construct $T$ in Magma \cite{magma} via the \texttt{AutomorphismGroupSimpleGroup} and \texttt{Socle} functions, and show that (ii) does not hold via fast, direct computations.

Suppose next that $T = \mathrm{Th}$, let $C$ be a proper subgroup of $T$, and let $A$ be a subgroup of $C$. Observe that if $b(T,C) = 2$, then (as discussed below the statement of Theorem~\ref{thm:intersecting}) the two-point stabiliser $C \cap C^t$ is trivial for some $t$. Hence $A \cap A^t = 1$, and so $A$ does not satisfy (ii). In general, $b(T,A) \le b(T,C)$. By \cite[Theorem 1]{burnesssporadic}, $T$ has only two maximal subgroups (up to conjugacy) with corresponding base size greater than two, namely $M_1:={}^3D_4(2) \sd 3$ and $M_2:=2^5.\PSL_5(2)$, and $b(T,M_1) = b(T,M_2) = 3$. Hence any proper subgroup $A$ of $T$ satisfying (ii) is a non-simple subgroup of $M_1$ or $M_2$ with $b(T,A) = 3$.

Now, let $\Sigma$ be the set of right cosets in $T$ of a proper subgroup $A$, and let $c$ be the number of orbits of $A$ on $\Sigma$. If $b(T,A) \ge 3$, then $|A \cap A^t| \ge 2$ for all $t \in T$, and it follows from the Orbit-Stabiliser Theorem that $c|A|/2 \ge |\Sigma|$. If $A$ is a maximal subgroup of either $M_1$ or $M_1' \cong {}^3D_4(2)$, then the character table of $A$ is included in The GAP Character Table Library. Thus for every such $A \ne M_1'$, we can compute $c$ as in the proof of Proposition~\ref{prop:sporadicgps}, and we see that in fact $c|A|/2 < |\Sigma|$, and hence $b(T,A) = 2$. To show that $b(T,A) = 2$ for each maximal subgroup $A$ of $M_2$, we construct $T$ and $M_2$ in Magma as subgroups of $\GL_{248}(2)$ using the respective generating pairs $\{x,y\}$ and $\{w_1,w_2\}$ given in \cite{onlineatlas}. Magma calculations (with a runtime of less than 10 minutes and a memory usage of 1.5 GB) then show that $A \cap A^y = 1$ for all representatives $A$ of the conjugacy classes of maximal subgroups of $M_2$ returned by the \texttt{MaximalSubgroups} function. Therefore, $M_1$ and $M_2$ are the only non-simple subgroups of $T$ with corresponding base size at least three. Further character table computations in GAP show that if $A \in \{M_1,M_2\}$, then for each proper normal subgroup $B$ of $A$, the number of $B$-orbits on $\Sigma$ is greater than the number of $A$-orbits. An additional application of the Orbit-Stabiliser Theorem yields $B(A \cap A^t) < A$ for some $t \in T$. Therefore, (ii) does not hold.

Finally, suppose that $T = \mathbb{M}$. As above, any proper subgroup $A$ of $T$ satisfying (ii) also satisfies $b(T,A) \ge 3$. By \cite[Theorem 3.1]{burnesssoluble}, the unique (up to conjugacy) proper subgroup $A$ of $T$ with $b(T,A) \ge 3$ is the maximal subgroup $K:=2.\mathbb{B}$, with $b(T,K) = 3$. Since $K$ is quasisimple, its centre $Z$ of order two is its unique nontrivial proper normal subgroup, and $Z$ lies in each maximal subgroup of $K$. Hence $Z(K \cap K^t) < K$ for all $t \in T \setminus K$, and so (ii) does not hold.
\end{proof}

We now prove our main theorem.

\begin{proof}[Proof of Theorem~\ref{thm:spreaddiag}]
As each spreading primitive group is synchronising, Theorem~\ref{thm:synchdiag} implies that a spreading primitive group of diagonal type has socle $T \times T$, for some non-abelian finite simple group $T$. For each $T$, let $\Omega := T$, and let $W(T)$ be the subgroup of $\Sym(\Omega)$ defined in \S\ref{sec:prel}. Recall also that each subgroup of a non-spreading group is non-spreading. Thus it suffices to show that $W(T)$ is non-spreading for each $T$. If $T \notin \{\mathrm{J}_1,\mathrm{M}_{22}, \mathrm{J}_3, \mathrm{McL},\mathrm{Th},\mathbb{M}\}$, then this is an immediate consequence of Theorem~\ref{thm:intersecting} and Corollary~\ref{cor:diagonalaction}.

For each of the six remaining groups $T$, let $r$, $s_1$ and $s_2$ be members of the conjugacy classes of $T$ given in Table~\ref{table:conjclasses}. By inspecting the character table for $T$ in \cite{ATLAS}, we observe that these elements satisfy Properties (i) and (ii) of Lemma~\ref{lem:chars}. Therefore, that lemma shows that $W(T)$ is non-spreading.
\end{proof}

\begin{table}
\centering
\renewcommand{\arraystretch}{1.2}
\caption{Elements $r$, $s_1$ and $s_2$ that satisfy Properties (i) and (ii) of Lemma~\ref{lem:chars}, for each of six sporadic groups $T$. Each element is specified by its corresponding conjugacy class, labelled as in the \textsc{Atlas} \cite{ATLAS}.}
\label{table:conjclasses}
\begin{tabular}{ cccc }
\hline
$T$ & $r$ & $s_1$ & $s_2$ \\
\hline
$\mathrm{J}_1$ & 7A & 5A & 5B \\
$\mathrm{M}_{22}$ & 5A & 7A & 7B \\
$\mathrm{J}_3$ & 5A & 19A & 19B \\
$\mathrm{McL}$ & 4A & 9A & 9B \\
$\mathrm{Th}$ & 7A & 39A & 39B \\
$\mathbb{M}$ & 110A & 119A & 119B \\
\hline
\end{tabular}
\end{table}

\subsection*{Acknowledgements} This work forms part of an Australian Research Council Discovery Project
DP200101951. We thank Tim Burness for drawing our attention to the result \cite[Theorem 3.1]{burnesssoluble}, which helped us to resolve the $\mathbb{M}$ case of Theorem~\ref{thm:intersecting}. We also thank the referee for some helpful suggestions.

\bibliographystyle{plain}
\bibliography{Synch}

\begin{thebibliography}{10}

\bibitem{synchacs}
Jo{\~{a}}o Ara\'{u}jo, Peter~J. Cameron, and Benjamin Steinberg.
\newblock Between primitive and 2-transitive: synchronization and its friends.
\newblock {\em EMS Surv. Math. Sci.}, 4(2):101--184, 2017.

\bibitem{arnold}
Fredrick Arnold and Benjamin Steinberg.
\newblock Synchronizing groups and automata.
\newblock {\em Theoret. Comput. Sci.}, 359(1-3):101--110, 2006.

\bibitem{polarspaces}
John Bamberg, Michael Giudici, Jesse Lansdown, and Gordon~F. Royle.
\newblock Tactical decompositions in finite polar spaces and non-spreading
  classical group actions (in preparation).

\bibitem{bglrdiag}
John Bamberg, Michael Giudici, Jesse Lansdown, and Gordon~F. Royle.
\newblock Synchronising primitive groups of diagonal type exist.
\newblock {\em Bull. Lond. Math. Soc.}, 54(3):1131--1144, 2022.

\bibitem{rank3paper}
John Bamberg, Michael Giudici, Jesse Lansdown, and Gordon~F. Royle.
\newblock Separating rank 3 graphs.
\newblock {\em European J. Combin.}, 112:Paper No. 103732, 13, 2023.

\bibitem{coherent}
John Bamberg and Jesse Lansdown.
\newblock The synchronisation hierarchy for coherent configurations (in
  preparation).

\bibitem{magma}
Wieb Bosma, John Cannon, and Catherine Playoust.
\newblock The {M}agma algebra system. {I}. {T}he user language.
\newblock {\em J. Symbolic Comput.}, 24(3-4):235--265, 1997.

\bibitem{bccsz}
John~N. Bray, Qi~Cai, Peter~J. Cameron, Pablo Spiga, and Hua Zhang.
\newblock The {H}all-{P}aige conjecture, and synchronization for affine and
  diagonal groups.
\newblock {\em J. Algebra}, 545:27--42, 2020.

\bibitem{GAPchar}
Thomas Breuer.
\newblock \textit{The GAP Character Table Library, Version 1.3.4}, 2022.

\bibitem{burnesssoluble}
Timothy~C. Burness.
\newblock On soluble subgroups of sporadic groups.
\newblock {\em Israel J. Math.}, 254(1):313--340, 2023.

\bibitem{burnesssporadic}
Timothy~C. Burness, E.~A. O'Brien, and Robert~A. Wilson.
\newblock Base sizes for sporadic simple groups.
\newblock {\em Israel J. Math.}, 177:307--333, 2010.

\bibitem{carter}
Roger~W. Carter.
\newblock {\em Simple groups of {L}ie type}, volume Vol. 28 of {\em Pure and
  Applied Mathematics}.
\newblock John Wiley \& Sons, London-New York-Sydney, 1972.

\bibitem{ATLAS}
J.~H. Conway, R.~T. Curtis, S.~P. Norton, R.~A. Parker, and R.~A. Wilson.
\newblock {\em Atlas of finite groups}.
\newblock Oxford University Press, Eynsham, 1985.
\newblock Maximal subgroups and ordinary characters for simple groups, With
  computational assistance from J. G. Thackray.

\bibitem{rationalelts}
Walter Feit and Gary~M. Seitz.
\newblock On finite rational groups and related topics.
\newblock {\em Illinois J. Math.}, 33(1):103--131, 1989.

\bibitem{GAP}
The GAP~Group.
\newblock {\em {GAP -- Groups, Algorithms, and Programming, Version 4.12.2}},
  2022.

\bibitem{isaacs}
I.~Martin Isaacs.
\newblock {\em Character theory of finite groups}, volume No. 69 of {\em Pure
  and Applied Mathematics}.
\newblock Academic Press [Harcourt Brace Jovanovich, Publishers], New
  York-London, 1976.

\bibitem{lps}
Martin~W. Liebeck, Cheryl~E. Praeger, and Jan Saxl.
\newblock On the {O}'{N}an-{S}cott theorem for finite primitive permutation
  groups.
\newblock {\em J. Austral. Math. Soc. Ser. A}, 44(3):389--396, 1988.

\bibitem{realelts}
Pham~Huu Tiep and A.~E. Zalesski.
\newblock Real conjugacy classes in algebraic groups and finite groups of {L}ie
  type.
\newblock {\em J. Group Theory}, 8(3):291--315, 2005.

\bibitem{onlineatlas}
Robert Wilson, Peter Walsh, Jonathan Tripp, Ibrahim Suleiman, Richard Parker,
  Simon Norton, Simon Nickerson, Steve Linton, John Bray, and Rachel Abbott.
\newblock \emph{ATLAS of Finite Group Representations - Version~3}.
\newblock \url{http://atlas.math.rwth-aachen.de}.
\newblock Accessed: August 22, 2023.

\end{thebibliography}

\end{document}